\documentclass[11pt]{amsart}
\usepackage{graphicx}
\usepackage{oldgerm}
\usepackage{mathdots}  
\usepackage{stmaryrd}
\usepackage{bm}
\usepackage[all]{xy}
\usepackage{color}
\usepackage{enumerate}

\usepackage{hyperref}
\usepackage{mathrsfs}
\usepackage{tikz} 
\usetikzlibrary{arrows,%
  decorations.markings,%
  matrix,%
  shapes}

\usepackage[latin1]{inputenc}
\usepackage{amsmath,amsthm, amscd, amssymb, amsfonts}

\newtheorem{theoremA}{Theorem}
\numberwithin{equation}{section}

\theoremstyle{plain}
\numberwithin{equation}{section}

\newtheorem{theorem}[equation]{Theorem}
\newtheorem{corollary}[equation]{Corollary}

\newtheorem{lemma}[equation]{Lemma}

\theoremstyle{definition}

\newtheorem{example}[equation]{Example}

\theoremstyle{definition}
\newtheorem{remark}[equation]{Remark}

\theoremstyle{definition}

\newcommand{\C}{\mathscr{C}}

\newcommand\id{\operatorname{id}}

\newcommand\Ext{\operatorname{Ext}}
\newcommand\Tor{\operatorname{Tor}}

\newcommand\ev{\operatorname{ev}}
\newcommand\odd{\operatorname{odd}}
\newcommand\lcm{\operatorname{lcm}}

\newcommand\op{\operatorname{op}}
\newcommand\ot{\otimes}
\renewcommand\mod{\operatorname{mod}}
\newcommand\Hom{\operatorname{Hom}}
\newcommand\stHom{\operatorname{\underline{Hom}}}

\newcommand\Hoch{\operatorname{HH}}
\newcommand\Der{\operatorname{D}}
\newcommand\HH{\Hoch}

\newcommand{\Coh}{\operatorname{H}\nolimits}

\newcommand{\az}{\mathfrak{a}}

\newcommand{\m}{\mathfrak{m}}
\newcommand{\p}{\mathfrak{p}}

\newcommand{\diagram}[3]{\matrix (#1) [matrix of math nodes,row
  sep={#2},column sep={#3},text height=1.5ex,text
  depth=0.25ex]}

\makeatletter
\def\blx@maxline{77}
\makeatother

\begin{document}
\title{Asymptotic vanishing of cohomology in triangulated categories}

\author{Petter Andreas Bergh, David A.\ Jorgensen, Peder Thompson}

\address[P. A. Bergh]{Institutt for matematiske fag, NTNU, N-7491 Trondheim, Norway}
\email{petter.bergh@ntnu.no}
\urladdr{https://www.ntnu.edu/employees/petter.bergh}
\address[D. A. Jorgensen]{Department of Mathematics, University of Texas at Arlington, 411 S. Nedderman Drive, Pickard Hall 429, Arlington, TX 76019, U.S.A.}
\email{djorgens@uta.edu}
\urladdr{http://www.uta.edu/faculty/djorgens/}
\address[P. Thompson]{Division of Mathematics and Physics, M{\"a}lardalen University, V{\"a}ster{\aa}s, Sweden}
\email{peder.thompson@mdu.se}
\urladdr{https://www.mdu.se/en/malardalen-university/staff?id=ptn04}

%\thanks{Part of this work was completed at Centro Internazionale per la Ricerca Matematica in Trento, Italy; all three authors are grateful for their support.}
%\date{July 24, 2023}

\begin{abstract}
Given a graded-commutative ring acting centrally on a triangulated category, our main result shows that if cohomology of a pair of objects of the triangulated category is finitely generated over the ring acting centrally, then the asymptotic vanishing of the cohomology is well-behaved.  In particular, enough consecutive asymptotic vanishing of cohomology implies all eventual vanishing.  Several key applications are also given. 
\end{abstract}

\maketitle

\section{Introduction}
The vanishing of homology and cohomology is a basic consideration in algebra. There have been many recent accounts of asymptotic vanishing of homology over certain commutative rings in the literature, essentially beginning with  \cite{Auslander1961}; see also \cite{Murthy1963}, \cite{Lichtenbaum1966}, \cite{HunekeWiegand1994}, \cite{Jorgensen2001}, \cite{HunekeJorgensenWiegand2001}, 
\cite{BerghJorgensen2011}, \cite{Celikbas2011}, \cite{Bergh2013}, \cite{Dao2013}, \cite{CelikbasIyengarPiepmeyerWiegand2015}. Several recent articles have turned to investigations of asymptotic vanishing of cohomology \cite{Jorgensen1999}, \cite{Bergh2008}, \cite{Bergh2011}, \cite{CelikbasDao2011}, \cite{AvramovIyengar2008}. A common thread to these articles is that they have made some appeals to a Noetherian condition on the totality of cohomology.  In this paper we distill these concepts into the context of triangulated categories, where cohomology---unlike homology---is naturally defined, and plays a prominent role.  We show that if the triangulated category admits a central ring action over which the cohomology is eventually Noetherian, then asymptotic vanishing is rather well-behaved, in the sense that sporadic vanishing cannot occur.  Specifically, our main result is the following.

\begin{theoremA}
Let $( \C, \Sigma, \Delta )$ be a triangulated category with a central ring action from a graded-commutative non-negatively graded ring $H$, and $A,B \in \C$ two objects such that the $H$-module $\Hom_{\C}^{\ge n_0}(A,B)$ is Noetherian for some $n_0 \in \mathbb Z$. Assume that one of the following holds:
\begin{enumerate}
\item[\emph{(1)}] $\Hom_{\C}(A, \Sigma^n B)$ has finite length over $H^0$ for $n \gg 0$;
\item[\emph{(2)}] $H^0$ contains an infinite field, or is local with an infinite residue field.
\end{enumerate}
Then either $\Hom_{\C}(A, \Sigma^n B) = 0$ for all $n\gg 0$, or there exist integers $d>0$ and $m_0$ 
such that $\Hom_{\C}(A, \Sigma^n B) \ne 0$ for all $n \ge m_0$ with 
$n \equiv j \hspace{1mm} ( \mod d)$, for at least one $j$ with $0\le j<d$.
\end{theoremA}

This is Theorem \ref{thm:main} from Section \ref{sec:main}. This section also contains notation and assumptions used throughout the remainder of the paper, as well as some key lemmas.  In Section \ref{sec:applications} we give several important applications of the main result, specifically to cohomology of pairs of finitely generated modules over the following kinds of rings: (1) commutative complete intersection rings, (2) certain finite-dimensional algebras, including quantum complete intersections and exterior algebras, and (3) modular group algebras.  

\subsection*{Acknowledgements}
Part of this work was completed at Centro Internazionale per la Ricerca Matematica in Trento, Italy; all three authors are grateful for their support. We would also like to thank the referee for valuable suggestions and comments that improved the presentation of the paper.

\section{The main result}\label{sec:main}

Throughout this section, we fix a triangulated category $( \C, \Sigma, \Delta )$. Thus $\C$ is an additive category equipped with an automorphism $\Sigma$, and $\Delta$ is a collection of triangles with respect to $\Sigma$---called distinguished triangles---satisfying Verdier's axioms \cite[Chapitre II, D{\'e}finition 1.1.1]{Verdier1996}. We shall often just refer to $\C$ instead of the more correct $( \C, \Sigma, \Delta )$. For objects $A,B\in\C$ we set 
$\Hom_\C^*(A,B)=\oplus_{n\in\mathbb Z}\Hom_\C(A,\Sigma^nB)$. Recall for any object $A\in\C$ the 
$\mathbb Z$-graded abelian group $\Hom_\C^*(A,A)$ is a graded ring under composition products.

A \emph{central ring action} on $\C$ is a non-negatively graded ring 
$H = \oplus_{n = 0}^{\infty} H^n$ which is graded-commutative, together with a ring homomorphism
\begin{center}
\begin{tikzpicture}
\diagram{d}{3em}{3em}{
 H & \Hom_\C^*(A,A) \\
 };
\path[->, font = \scriptsize, auto]
(d-1-1) edge node{$\varphi_A$} (d-1-2);
\end{tikzpicture}
\end{center}
of graded rings for every object $A\in\C$, such that for all $B\in\C$ the left and right actions of $H$ on 
$\Hom_{\C}^*(A, B)$ via $\varphi_A$ and $\varphi_B$ are compatible. Specifically, for 
$x \in H^m$ and $f \in \Hom_{\C}(A, \Sigma^n B)$ we have 
$\varphi_B(x) \cdot f = (-1)^{mn} f \cdot \varphi_A(x)$. We shall therefore just refer to 
$\Hom_{\C}^*(A,B)$ as a graded $H$-module.

We now prove a couple of general lemmas on graded rings and modules. The first one is known, but we include a proof; see also \cite[Section 2.1]{BerghIyengarKrauseOppermann2010} and \cite[Section 1.5]{BrunsHerzog1993}. For a non-negatively graded ring $S$, we shall denote the even part, that is, the subalgebra $\oplus_{n = 0}^{\infty} S^{2n}$, by $S^{\ev}$. Furthermore, we shall denote the positive part of $S$, that is, the ideal $\oplus_{n = 1}^{\infty} S^{n}$, by $S^+$. Finally, given a graded 
$S$-module $M$ and an integer $n_0$, we shall denote the submodule $\oplus_{n = n_0}^{\infty} M^n$ by $M^{\ge n_0}$. Thus, for example, the $S$-module $\oplus_{n = 0}^{\infty} M^n$ is denoted by 
$M^{\ge 0}$.

\begin{lemma}\label{lem:even}
If $S$ is a non-negatively graded ring which is graded-commutative, then the following are equivalent:

\emph{(1)} $S$ is Noetherian;

\emph{(2)} $S^{\ev}$ is Noetherian, and $S$ is finitely generated as a module over this subalgebra.
\end{lemma}

\begin{proof}
If (1) holds, consider the homogeneous ideal in $S$ generated by the positive part of 
$S^{\ev}$. As $S$ is Noetherian, this ideal is finitely generated, say by homogeneous elements 
$x_1, \dots, x_t \in S^{\ev}$ of positive degrees. A standard argument (cf.\ the proof of \cite[Proposition 10.7]{AtiyahMacdonald1969}) now shows that $S^{\ev}$ is finitely generated as an algebra over 
$S^0$ by $x_1, \dots, x_t$. Since $S^0 \simeq S / S^+$ is Noetherian, then so is $S^{\ev}$. Next, consider the odd part $S^{\odd} = \oplus_{n = 1}^{\infty} S^{2n-1}$ of $S$; this is a module over 
$S^{\ev}$. As above, the homogeneous ideal in $S$ generated by $S^{\odd}$ is finitely generated, say by homogeneous elements $y_1, \dots, y_s \in S^{\odd}$. These elements generate $S^{\odd}$ as an 
$S^{\ev}$-module, hence $S = S^{\ev} \oplus S^{\odd}$ is finitely generated as an $S^{\ev}$-module. Thus (2) holds, and the converse is obvious.
\end{proof}

The following lemma is a variant of \cite[Theorem 2.5]{Bergh2006} and \cite[Theorem 6.5]{BerghJorgensenThompson2023}. Given a homogeneous element $x$ in a graded ring, we denote by $\deg (x)$ its degree.

\begin{lemma}\label{lem:regular}
Let $S$ be a non-negatively graded Noetherian ring which is commutative in the ordinary sense, and $M$ a finitely generated graded $S$-module. Suppose that $S^0$ either contains an infinite field, or is local with an infinite residue field. Finally, choose homogeneous elements $x_1, \dots, x_t \in S$, of positive degrees, generating $S$ as an algebra over $S^0$, and set $d =  \lcm \{ \deg (x_1), \dots, \deg (x_t) \}$. Then there exists a homogeneous element $x \in S$, with $\deg (x) = d$, such that the multiplication map
\begin{center}
\begin{tikzpicture}
\diagram{d}{3em}{3em}{
 M^n & M^{n + d} \\
 };
\path[->, font = \scriptsize, auto]
(d-1-1) edge node{$\cdot x$} (d-1-2);
\end{tikzpicture}
\end{center}
is injective for $n \gg 0$.
\end{lemma}

\begin{proof}
By \cite[Lemma 2.5]{BerghIyengarKrauseOppermann2010}, there exists a positive integer $n_0$ and a homogeneous element $y \in S$, of positive degree, such that the corresponding multiplication map $M^n \longrightarrow M^{n + \deg (y)}$ is injective for $n \ge n_0$. Let $\p_1, \dots, \p_t$ be the associated prime ideals of the (finitely generated) $S$-module $M^{\ge n_0}$. As in the proof of \cite[Theorem 2.5]{Bergh2006}, $S^d$ is not contained in any of these prime ideals; if it were, say with $S^d \subseteq \p_j$, then each generator $x_i$, and hence also the element $y$, would belong to $\p_j$. But this is impossible since $y$ is regular on $M^{\ge n_0}$. Therefore $S^d \cap \p_i$ is strictly contained in $S^d$ for all $i$.

If $S^0$ contains an infinite field $k$, then $S^d \cap \p_i$ is a proper $k$-vector subspace of $S^d$ for all $i$. A vector space over an infinite field is never a finite union of proper subspaces, so the union
$$\left ( S^d \cap \p_1 \right ) \cup \cdots \cup \left ( S^d \cap \p_t \right )$$
must also be strictly contained in $S^d$. The same conclusion follows if $S^0$ is local with an infinite residue field. To see this, denote by $\m$ the maximal ideal of $S^0$, and $k'$ the residue field $S^0 / \m$. If we denote $S^d \cap \p_i$ by $V_i$, then it follows from Nakayama's lemma that $(V_i + \m S^d) / \m S^d$ is a proper $k'$-subspace of $S^d / \m S^d$ for all $i$. As above, the union 
$$\left ( (V_1 + \m S^d) / \m S^d \right ) \cup \cdots \cup \left ( (V_t + \m S^d) / \m S^d \right )$$
is then strictly contained in $S^d / \m S^d$, so that the union of the $V_i$, that is, the union
$$\left ( S^d \cap \p_1 \right ) \cup \cdots \cup \left ( S^d \cap \p_t \right )$$
must be strictly contained in $S^d$.

It follows from what we have just proved that there exists an element $x \in S^d$ which is not contained in the union $\p_1 \cup \cdots \cup \p_t$. This union is the set of zero-divisors in $S$ of the module $M^{\ge n_0}$, and so $x$ must be regular on this module. Therefore the multiplication map on $M^{\ge n_0}$ induced by $x$ is injective.
\end{proof}

Before getting to the main theorem, we include the following remark.

\begin{remark}\label{rem:facts} 
Let $H$ be a graded-commutative and non-negatively graded ring acting centrally on $\C$. In the present passage we assume that for certain objects  $A, B\in\C$ the $H$-module
$\Hom_\C^*(A,B)$ is eventually Noetherian, that is, there exists $n_0\in\mathbb Z$ such that
$M=\Hom_\C^{\ge n_0}(A,B)$ is Noetherian. Denote by $\az$ the annihilator ideal of $M$ in $H$. This is a homogeneous ideal, and $M$ is finitely generated as a module over the graded ring $H/\az$. Moreover, this ring is Noetherian. To see this, let $m_1, \dots, m_t$ be homogeneous elements that generate $M$ over $H$, and consider the $H$-module homomorphism 
$H\longrightarrow M \oplus \cdots \oplus M$ (with $t$ copies of $M$ in the direct sum) given by 
$s \mapsto (sm_1, \dots, sm_t)$. The kernel is a homogeneous ideal of $H$, and must equal $\az$, since the ring is graded-commutative. We therefore obtain an injective $H$-homomorphism 
$H/\az \longrightarrow M \oplus \cdots \oplus M$, and so $H/\az$ is a Noetherian $H$-module. This implies that $H/\az$ is Noetherian as a ring. Let $E$ denote the even part of $H/\az$, that is,
\[
E=(H/\az)^{\ev}
\]
Then Lemma \ref{lem:even} says that $H/\az$ is module-finite over $E$ and it follows that $M$ is a finitely generated graded module over the commutative Noetherian graded ring $E$. Choose elements 
$x_1,\dots,x_t\in E$ of positive degrees generating $E$ as an algebra over $E^0$ and set 
\[
d=d(x_1,\dots,x_t)
\] 
to be the least common multiple of
the set $\{\deg(x_1),\dots,\deg(x_t)\}$.
\end{remark}

We now give the main theorem. 

\begin{theorem}\label{thm:main}
Let $( \C, \Sigma, \Delta )$ be a triangulated category with a central ring action from a graded-commutative non-negatively graded ring $H$, and $A,B \in \C$ two objects such that the $H$-module $\Hom_{\C}^{\ge n_0}(A, B)$ is Noetherian for some $n_0\in\mathbb Z$. Assume that one of the following holds:
\begin{enumerate}
\item[\emph{(1)}] $\Hom_{\C}(A, \Sigma^n B)$ has finite length over $H^0$ for $n \gg 0$;
\item[\emph{(2)}] $H^0$ contains an infinite field, or is local with an infinite residue field.
\end{enumerate}
Then either $\Hom_{\C}(A, \Sigma^n B) = 0$ for all $n\gg 0$, or there exist integers $d>0$ and $m_0$ 
such that $\Hom_{\C}(A, \Sigma^n B) \ne 0$ for all $n \ge m_0$ with 
$n \equiv j \hspace{1mm} ( \mod d)$, for at least one $j$ with $0\le j<d$.
\end{theorem}

\begin{proof}
Suppose first that (1) holds. Without loss of generality we can assume that $n_0$ satisfies 
$\Hom_{\C}^{\ge n_0}(A, B)$ is Noetherian and $\Hom_{\C}(A, \Sigma^n B)$ has finite length over 
$H^0$ for $n \ge n_0$. It follows that $\Hom_{\C}(A, \Sigma^n B)$ also has finite length over 
$(H/\az)^0$ for $n \ge n_0$, where $\az$ is the annihilator of $\Hom_{\C}^{\ge n_0}(A, B)$ in $H$. We may now replace the object $B$ by $\Sigma^{n_0}B$, since this does not change the conclusion, and assume that $n_0 = 0$. Thus we have that $\Hom_\C^{\ge 0}(A,B)$ is finitely generated over the Noetherian ring $H/\az$ and $\Hom_{\C}(A, \Sigma^n B)$ has finite length over $(H/\az)^0$ for 
$n \ge 0$. 

As in Remark \ref{rem:facts}, $\Hom_{\C}^{\ge 0}(A,B)$ is finitely generated as a module over the Noetherian ring $E$, where $E$ is the even part of $H/\az$. Note that $E^0=(H/\az)^0$. By the classical Hilbert-Serre theorem (cf.\ \cite[Theorem 11.1]{AtiyahMacdonald1969}), there is an equality of power series 
\[
\sum_{n=0}^{\infty} \ell_{E^0} \left ( \Hom_{\C}(A, \Sigma^n B) \right ) z^n = \frac{g(z)}{\prod_{i=1}^t \left (1-z^{\deg (x_i)} \right )}
\]
where $g(z)$ is a polynomial with integer coefficients, $\ell_{E^0}$ denotes the length over $E^0$ and the $x_i$ are elements of $E^{>0}$ defined in Remark \ref{rem:facts}. Thus 
$x_1,\dots,x_t$ are elements of $E$ of positive degrees generating $E$ over $E^0$ and 
$d=d(x_1,\dots,x_t)$ is the least common multiple of the degrees of the $x_i$. We may now assume that $t \ge 1$, that is, that the sequence $x_1, \dots, x_t$ is not empty; otherwise $E=E^0$, and 
$\Hom_{\C}(A, \Sigma^n B) = 0$ for $n \gg 0$ automatically. Write the right-hand side as
$$u(z) + \frac{v(z)}{\prod_{i=1}^t \left (1-z^{\deg (x_i)} \right )}$$
for polynomials $u(z), v(z)$ with integer coefficients, and where the degree of $v(z)$ is strictly less than $\sum_{i=1}^t \deg (x_i)$, that is, the degree of the denominator. By cancelling common factors in $v(z)$ and the denominator, we obtain an equality

\[
\sum_{n=0}^{\infty} \ell_{E^0} \left ( \Hom_{\C}(A, \Sigma^n B) \right ) z^n = u(z) +  \frac{p(z)}{q(z)}
\]
where the fraction on the right-hand side is reduced, and the degree of $p(z)$ is strictly less than that of $q(z)$. Moreover, every root $\alpha$ of $q(z)$ satisfies $\alpha^d =1$, since $\alpha^{\deg (x_i)} =1$ for some $i$.

It now follows from \cite[Proposition 4.4.1]{Stanley2012} that there exist polynomials 
$g_1(y), \dots, g_d(y)$ in $\mathbb{C}[y]$ with the following property: for every $n$ greater than the degree of $u(z)$, there is an equality
\[
\ell_{E^0} \left ( \Hom_{\C}(A, \Sigma^n B) \right ) = g_i(n)
\]
when $n \equiv i \hspace{1mm} ( \mod d)$. If all the $g_i$ are zero, then 
$\Hom_{\C}(A, \Sigma^n B) =0$ for all $n \gg 0$, and we are done. If not, let 
$\alpha_1, \dots, \alpha_s$ be the roots of the nonzero polynomials among $g_1(y), \dots, g_d(y)$, and take any integer $m_0$ greater than the maximum of the $| \alpha_i |$, and larger than the degree of $u(z)$. Moreover, choose an index $j$ with $g_j$ nonzero. Then for any $n \ge m_0$ with $n \equiv j \hspace{1mm} ( \mod d)$, we see that
\[
\ell_{E^0} \left ( \Hom_{\C}(A, \Sigma^n B) \right ) = g_j(n) \neq 0
\]
This proves the theorem when (1) holds.

Next, suppose that (2) holds. As before, we follow the notation from Remark \ref{rem:facts}, letting $E = (H / \az)^{\ev}$. If $\Hom_\C^{\ge n_0}(A,B)=0$ then we are done. Otherwise, the annihilator $\az$ of $\Hom_\C^{\ge n_0}(A,B)$ in $H$ is not the unit ideal. Therefore, if
$H^0$ contains an infinite field or is local with infinite residue field, then the same is true for
$E^0=(H/\az)^0$. According to Remark \ref{rem:facts} we may then use Lemma \ref{lem:regular} directly with 
$S$ replaced by $E$: there exists a homogeneous element $x \in E^d$ with the property that the multiplication map
\begin{center}
\begin{tikzpicture}
\diagram{d}{3em}{3em}{
 \Hom_{\C}(A, \Sigma^n B) & \Hom_{\C}(A, \Sigma^{n+d} B) \\
 };
\path[->, font = \scriptsize, auto]
(d-1-1) edge node{$\cdot x$} (d-1-2);
\end{tikzpicture}
\end{center}
is injective for all $n\gg 0$. Now if $\Hom_\C(A,\Sigma^n B)=0$ for all $n\gg 0$ we are done. If not, there exists an integer $m_0$ such that multiplication by $x$ is injective for all 
$n\ge m_0$ and $\Hom_\C(A,\Sigma^{m_0}B)\ne 0$. By injectivity of $x$ again it follows that 
$\Hom_\C^n(A,\Sigma^nB)\ne 0$ for all $n\ge m_0$ with $n \equiv m_0 \hspace{1mm} ( \mod d)$. This proves the theorem, with $j$ being the residue of $m_0$ modulo $d$.
\end{proof}

\begin{remark}\label{rem:main}
(1) A potentially more useful equivalent formulation of the conclusion of Theorem 
\ref{thm:main} is the following: there exists an integer $d>0$ such that 
$\Hom_{\C}(A, \Sigma^n B) = 0$ 
for all $n\gg 0$  if and only if for each $j$ with $0\le j<d$ we have $\Hom_\C^n(A,\Sigma^nB)=0$ for infinitely many $n>0$ with $n \equiv j \hspace{1mm} ( \mod d)$.

(2) If $A,B$ and $H$ are as in the theorem, let $\az$ be the annihilator ideal of $\Hom_{\C}^{\ge n_0}(A, B)$ in $H$. From Remark \ref{rem:facts}, we see that the integer $d$ can be chosen to be the least common multiple of the degrees of the elements in a generating set of the Noetherian ring $(H/\az)^{\ev}$ over its degree zero part $(H/\az)^0$.
\end{remark}

\section{Applications}\label{sec:applications}

In this section, we apply the main result of Section \ref{sec:main} to the cohomology of modules over certain rings. We start with commutative local complete intersections. Recall that such a ring $A$ is a commutative Noetherian local ring, whose completion (with respect to the maximal ideal) is of the form $S / (x_1, \dots, x_c)$ for some regular local ring $S$, and a regular sequence $x_1, \dots, x_c$ contained in the maximal ideal of $S$. When this is the case, then there always exists such a regular sequence contained in the square of the maximal ideal of $S$, and the length $c$ is then unique. This is the \emph{codimension} of the local complete intersection $A$. Note that regular local rings are precisely the local complete intersections of codimension zero.

The following is a generalization of \cite[Lemma 3.2]{Jorgensen1999}. We denote the Krull dimension of $A$ by $\dim A$.

\begin{theorem}\label{thm:ci}
Let $A$ be a commutative local complete intersection ring, and $M,N$ two finitely generated $A$-modules. Then either $\Ext_A^{n}(M,N) = 0$ for all $n > \dim A$, or there exists a non-negative integer $m_0$ with $\Ext_A^{n}(M,N) \neq 0$ for all even or all odd integers $n \ge m_0$.
\end{theorem}

\begin{proof}
Denote the codimension of $A$ by $c$. By \cite[Lemma 2.2]{AvramovBuchweitz2000}, there exists a (faithfully) flat extension $A \longrightarrow B$ of local rings, where $B$ is a codimension $c$ complete intersection with the following properties: the residue field of $B$ is infinite, and $B = S / (x_1, \dots, x_c)$ for some regular local ring $S$, and a regular sequence $x_1, \dots, x_c$ contained in the maximal ideal of $S$. The Krull dimension of $B$ is the same as that for $A$, and it follows from \cite[Proposition 2.5.8]{Grothendieck1965} that there is an isomorphism
$$\Ext_B^n( B \otimes_A M, B \otimes_A N) \simeq B \otimes_A \Ext_A^{n}(M,N)$$
for every $n$. We may therefore assume that $A$ itself is of the form $S / (x_1, \dots, x_c)$, and with an infinite residue field.

Consider the bounded derived category $\Der^b(A)$ of $A$. As explained for example in \cite[Section 7.1]{AvramovIyengar2007} or \cite[Section 11]{BensonIyengarKrause2008}, there exists a polynomial ring $A[ \chi_1, \dots, \chi_c ]$ of cohomology operators of degree two---called Eisenbud operators---acting centrally on $\Der^b(A)$, in such a way that $\Hom_{\Der^b(A)}^{\ge 0}(X,Y)$ is finitely generated for all $X,Y \in \Der^b(A)$. Moreover, as $A$ has an infinite residue field, we see that condition (2) in the statement of Theorem \ref{thm:main} is satisfied. Now use the theorem, Remark \ref{rem:main}(2), and the fact that there is an isomorphism $\Ext_A^{n}(M,N) \simeq \Hom_{\Der^b(A)}(M, \Sigma^nN)$ for every $n \ge 1$, when we view $M$ and $N$ as stalk complexes in the derived category: either $\Ext_A^{n}(M,N) = 0$ for $n \gg 0$, or there exists a non-negative integer $m_0$ with $\Ext_A^{n}(M,N) \neq 0$ for all even or all odd integers $n \ge m_0$. Finally, by \cite[Theorem III]{AvramovBuchweitz2000}, if $\Ext_A^{n}(M,N) = 0$ for all higher $n$, then $\Ext_A^{n}(M,N) = 0$ for all $n > \dim A$.
\end{proof}

There is also a homology version of the theorem, for $\Tor$; cf.\ \cite[Theorem 3.1]{Jorgensen1999}. One must then assume that the tensor product $M \ot_A N$ has finite length, in order to (in the proof) use Theorem \ref{thm:ci} via an isomorphism linking $\Ext$ and $\Tor$, involving the injective hull of the residue field. Specifically, we have ismomorphisms for all $i$
\[
\Ext^i_A(M,\Hom_A(N,E_A)) \simeq \Hom_A(\Tor_i^A(M,N),E_A)
\]
where $E_A$ is the injective hull of the residue field of $A$.

Let us now turn to finite-dimensional algebras. Given a field $k$ and a finite-dimensional $k$-algebra $A$, consider its Hochschild cohomology ring $\HH^*(A)$. The degree $n$ component of this ring is $\Ext_{A \otimes_k A^{\op}}^n(A,A)$, where $A^{\op}$ is the opposite algebra of $A$. In other words, the elements of the degree $n$ component $\HH^n(A)$ correspond to $n$-fold bimodule extensions of $A$ with itself. The multiplication is the Yoneda product, and it is well known that $\HH^*(A)$ is graded-commutative; cf.\ \cite{Gerstenhaber1963}. We now refer to \cite{SnashallSolberg2004} and \cite{Solberg2006} for details concerning what follows. Given a finitely generated left $A$-module $M$, the tensor product $- \ot_A M$ induces a homomorphism 
\begin{center}
\begin{tikzpicture}
\diagram{d}{3em}{3em}{
\HH^*(A) & \Ext_A^*(M,M) \\
 };
\path[->, font = \scriptsize, auto]
(d-1-1) edge node{$\varphi_M$} (d-1-2);
\end{tikzpicture}
\end{center}
of graded rings. Thus if $N$ is another finitely generated left $A$-module, then $\Ext_A^*(M,N)$ becomes a left $\HH^*(A)$-module via $\varphi_N$, and a right $\HH^*(A)$-module via $\varphi_M$. However, the two module structures coincide up to a sign, and so it does not matter whether we view $\Ext_A^*(M,N)$ as a left module or as a right module over $\HH^*(A)$. 

We can now state and prove the following asymptotic vanishing result for finite-dimensional algebras.

\begin{theorem}\label{thm:findim}
Let $A$ be a finite-dimensional algebra over a field, and $M,N$ two finitely generated left $A$-modules with the property that the $\HH^*(A)$-module $\Ext_A^{\ge n_0}(M,N)$ is Noetherian for some integer $n_0$. Then either $\Ext_A^{n}(M,N) = 0$ for all $n\gg 0$, or there exist integers $d>0$ and $m_0$ 
such that $\Ext_A^{n}(M,N) \ne 0$ for all $n \ge m_0$ with $n \equiv j \hspace{1mm} ( \mod d)$, for at least one $j$ with $0\le j<d$.
\end{theorem}

\begin{proof}
As explained in \cite[Section 10]{Solberg2006}, the Hochschild cohomology ring $\HH^*(A)$ acts centrally on the bounded derived category $\Der^b(A)$ of $A$. When we view $M$ and $N$ as stalk complexes in $\Der^b(A)$, then the action of $\HH^*(A)$ on $\Hom_{\Der^b(A)}^{\ge 0}(M,N)$ corresponds to the above action on $\Ext_A^*(M,N)$, via the two ring homomorphisms $\varphi_M$ and $\varphi_N$. Here, as in the proof of Theorem \ref{thm:ci}, we use the fact that there is an isomorphism $\Ext_A^{n}(M,N) \simeq \Hom_{\Der^b(A)}(M, \Sigma^nN)$ for every $n \ge 1$. As pointed out in \cite[Proposition 9.1]{BKSS2020}, the isomorphism $\Ext_A^{\ge n_0}(M,N) \simeq \Hom_{\Der^b(A)}^{\ge n_0}(M, N)$ is one of $\HH^*(A)$-modules, hence the $\HH^*(A)$-module $\Hom_{\Der^b(A)}^{\ge n_0}(M, N)$ is Noetherian.

Since $\HH^0(A)$ is just the center of $A$, and contains the ground field, we see that $\Hom_{\Der^b(A)}(M, \Sigma^n N)$ has finite length over $\HH^0(A)$ for all $n$. This means that condition (1) in the statement of Theorem \ref{thm:main} holds. Applying the theorem to $\Hom_{\Der^b(A)}^{\ge n_0}(M,N)$, we obtain what we want.
\end{proof}

For certain algebras one can give a more precise vanishing conclusion in the statement of the theorem, and also a more general version of the theorem itself. We explain all this in the following remark.

\begin{remark}\label{rem:findim}
(1) We say that a finite-dimensional algebra $A$ has \emph{finitely generated cohomology} if $\HH^*(A)$ is Noetherian, and $\Ext_A^*(M,N)$ is a finitely generated $\HH^*(A)$-module for all finitely generated left $A$-modules $M$ and $N$. Since all such $A$-modules have finite length, this is equivalent to the following: $\HH^*(A)$ is Noetherian, and $\Ext_A^*(A / \mathfrak{r}, A / \mathfrak{r} )$ is a finitely generated $\HH^*(A)$-module, where $\mathfrak{r}$ denotes the Jacobson radical of $A$.

If $A$ has finitely generated cohomology, then it is actually a Gorenstein algebra, by \cite[Proposition 2.2]{EHSST2004}. The injective dimension $\id_A A$ of $A$ as a left module over itself is then finite, say $\id_A A = s$. Given any finitely generated left $A$-module $M$, the syzygy module $M' = \Omega_A^s(M)$ is maximal Cohen-Macaulay, in the sense that $\Ext_A^n(M',A)=0$ for $n \ge 1$. If $\Ext_A^{n}(M,N) = 0$ for all higher $n$, then the same is true for $\Ext_A^n(M',N)$, and by combining this with \cite[Theorem 3.7]{Bergh2011} we see that the vanishing conclusion in the statement of Theorem \ref{thm:findim} can be sharpened considerably to the following: $\Ext_A^{n}(M,N) = 0$ for all $n > 2 \id_A A$.

(2) There are situations where it is more convenient to work with a suitably nice subalgebra of $\HH^*(A)$, and Theorem \ref{thm:findim} can be stated in this generality. Namely, let $H$ be a graded subalgebra of $\HH^*(A)$. Given a finitely generated left $A$-module $M$, the homomorphism $\varphi_M$ restricts to a homomorphism 
\begin{center}
\begin{tikzpicture}
\diagram{d}{3em}{3em}{
H & \Ext_A^*(M,M) \\
 };
\path[->, font = \scriptsize, auto]
(d-1-1) edge (d-1-2);
\end{tikzpicture}
\end{center}
of graded rings, and $H$ acts centrally on $\Der^b(A)$. The proof of Theorem \ref{thm:findim} now applies, with $\HH^{*}(A)$ replaced by $H$, and so we may replace $\HH^{*}(A)$ by $H$ in the statement of the theorem.

If $H$ is Noetherian and $\Ext_A^*(M,N)$ is a finitely generated $H$-module for all finitely generated left $A$-modules $M$ and $N$, then $\Hom_{\Der^b(A)}^{\ge 0}(X,Y)$ is a finitely generated $H$-module for all $X, Y \in \Der^b(A)$. Note that, by \cite[Proposition 5.7]{Solberg2006}, if there exists such a subalgebra $H$ of $\HH^*(A)$, then $A$ has finitely generated cohomology in the sense that we defined in the above remark, with respect to $\HH^*(A)$. In \emph{loc.\ cit.}, it is implicitly assumed that $H^0 = \HH^0(A)$, but this is not needed.
\end{remark}

Let us look at a large class of algebras for which the remark applies. Fix a field $k$, a nonzero element $q \in k$, and two positive integers $c \ge 1$ and $a \ge 2$. The corresponding \emph{quantum complete intersection} $A_{c,q}^a$ is the finite-dimensional $k$-algebra
$$A_{c,q}^a = k \langle x_1, \dots, x_c \rangle / (x_i^a, \{ x_ix_j - q x_jx_i \}_{i<j} )$$
It is commutative---and a local complete intersection---if $q=1$, but otherwise noncommutative when $c \ge 2$. We allow the case $c = 1$ here; in this case, the algebra is the commutative truncated polynomial ring $k[x]/(x^a)$, and the commutator element $q$ plays no role. Note that $A_{c,q}^a$ is always selfinjective. The following is a non-commutative version of Theorem \ref{thm:ci}.

\begin{theorem}\label{thm:qci}
Let $k$ be a field, and $c,a$ two positive integers with $c \ge 1$ and $a \ge 2$. Furthermore, let $q \in k$ be an element with $q^a =1$, and consider the corresponding quantum complete intersection $A = A_{c,q}^a$. Then, given two finitely generated left $A$-modules $M$ and $N$, either $\Ext_A^{n}(M,N) = 0$ for all $n > 0$, or there exists a non-negative integer $m_0$ with $\Ext_A^{n}(M,N) \neq 0$ for all even or all odd integers $n \ge m_0$.
\end{theorem}

\begin{proof}
If either $q = 1$ or $c=1$, then the algebra is isomorphic to the commutative local complete intersection ring $k \llbracket x_1, \dots, x_c \rrbracket / (x_1^a, \dots, x_c^a)$, where $k \llbracket x_1, \dots, x_c \rrbracket$ denotes the power series ring. In this case, Theorem \ref{thm:ci} applies, so from now on we assume that $q \neq 1$ and $c \ge 2$.

The algebra $A$ is local, with $k$ as the simple left module. Moreover, by \cite[Theorem 5.5]{BerghOppermann2008}, it has finitely generated cohomology, so Theorem \ref{thm:findim} and Remark \ref{rem:findim} apply. However, we shall show that there is a polynomial subalgebra of $\HH^*(A)$, with generators in degree two, that does the job for us, as in Remark \ref{rem:findim}(2).

First of all, it follows from \cite[Theorem 5.3]{BerghOppermann2008} that there exists a polynomial subalgebra $S = k[z_1, \dots, z_c]$ of the center of $\Ext_A^*(k,k)$, with the latter a finitely generated $S$-module. Now consider the natural $\mathbb{Z}^c$-grading on $A$ given by $\deg (x_i) = e_i$, where $e_i$ denotes the unit vector in $\mathbb{Z}^c$ with $1$ in the $i$th coordinate. This grading is passed down to $\Ext_A^*(k,k)$, so that each $\Ext_A^n(k,k)$ is $\mathbb{Z}^c$-graded as well. As explained in \cite[Section 3]{Oppermann2010}, the element $z_i$ belongs to $\Ext_A^{2, -ae_i}(k,k)$; in other words, it belongs to $\Ext_A^2(k,k)$, and has internal degree $-ae_i$.

\sloppy We now follow the notation from \cite[Section 2]{Oppermann2010}, in particular from \cite[Technical notation, page 825]{Oppermann2010}. Consider the vector $\mathbf{d} = -ae_i$. Then $Q \mathbf{d} = (1, \dots, 1)$, so that $( Q \mathbf{d} )_i =1$ for $1 \le i \le c$. Moreover, since $a \ge 2$, the three subsets $I_{\rm{max}}, I_1$ and $I_2$ of $\{ 1, \dots, c \}$, defined right before \cite[Theorem 3.4]{Oppermann2010} and relative to our vector $\mathbf{d}$, are just given by
$$I_{\rm{max}} = I_1 = \emptyset, \hspace{5mm} I_2 = \{ 1, \dots, c \}$$
It therefore follows from \cite[Theorem 3.5]{Oppermann2010} that the polynomial subalgebra $S$ of $\Ext_A^*(k,k)$ belongs to the image of the homomorphism
\begin{center}
\begin{tikzpicture}
\diagram{d}{3em}{3em}{
\HH^*(A) & \Ext_A^*(k,k) \\
 };
\path[->, font = \scriptsize, auto]
(d-1-1) edge node{$\varphi_k$} (d-1-2);
\end{tikzpicture}
\end{center}
Consequently, there exists a polynomial subalgebra $H = k[ \eta_1, \dots, \eta_c ]$ of $\HH^*(A)$, with $\deg ( \eta_i ) =2$ for all $i$, and with the property that $\Ext_A^*(k,k)$ is finitely generated when viewed as an $H$-module via $\varphi_k$. 

As explained in Remark \ref{rem:findim}(2), we may now apply Theorem \ref{thm:findim} with $\HH^*(A)$ replaced by $H$ in the statement, and with $n_0 =0$ since $A$ has finitely generated cohomology. Moreover, since our algebra is selfinjective, we see from Remark \ref{rem:findim}(1) that the vanishing conclusion actually becomes $\Ext_A^{n}(M,N) = 0$ for all $n > 0$. Finally, since the degree of each $\eta_i$ is two, it follows from Remark \ref{rem:main}(2) that we may choose $d=2$ in Theorem \ref{thm:findim}.
\end{proof}

Exterior algebras are particular examples of quantum complete intersections for which Theorem \ref{thm:qci} applies. We recored this in the following corollary.

\begin{corollary}\label{cor:exterior}
Let $k$ be a field, $c \ge 1$ an integer, and $A$ the corresponding exterior algebra
$$k \langle x_1, \dots, x_c \rangle / (x_i^2, \{ x_ix_j + x_jx_i \}_{i \neq j} )$$
Then, given two finitely generated left $A$-modules $M$ and $N$, either $\Ext_A^{n}(M,N) = 0$ for all $n > 0$, or there exists a non-negative integer $m_0$ with $\Ext_A^{n}(M,N) \neq 0$ for all even or all odd integers $n \ge m_0$.
\end{corollary} 

Our final application of the main result from Section \ref{sec:main} is the cohomology of modules over group algebras of finite groups. Let therefore $k$ be a field and $G$ a finite group. If the characteristic of $k$ does not divide the order of $G$, then the group algebra $kG$ is semisimple by Maschke's theorem. Therefore, from a homological point of view, the only interesting setting for us is when $k$ is of prime characteristic $p$, with $p$ dividing $|G|$. 

Recall that the cohomology ring $\Coh^*(G,k)$ of $G$ is the $\Ext$-algebra of the trivial $kG$-module, that is, $\Coh^*(G,k) = \Ext_{kG}^*(k,k)$. Given left $kG$-modules $M$ and $N$, the (Hopf algebra) tensor product $- \ot_k M$ induces a homomorphism
\begin{center}
\begin{tikzpicture}
\diagram{d}{3em}{3em}{
\Coh^*(G,k) & \Ext_{kG}^*(M,M) \\
 };
\path[->, font = \scriptsize, auto]
(d-1-1) edge (d-1-2);
\end{tikzpicture}
\end{center}
of graded rings, turning $\Ext_{kG}^*(M,N)$ into a left and a right $\Coh^*(G,k)$-module. As with the Hochschild cohomology ring, the group cohomology ring is graded-commutative, and the left and right module actions on $\Ext_{kG}^*(M,N)$ coincide up to a sign; cf.\ \cite[Section 3.2]{Benson1998}. By a classical result of Evens and Venkov (cf.\ \cite{Evens1961}, \cite{Venkov1959} and \cite{Venkov1961}), the cohomology ring is Noetherian, and $\Ext_{kG}^*(M,N)$ is finitely generated when $M$ and $N$ are finitely generated as $kG$-modules. 

\begin{theorem}\label{thm:grp}
Let $k$ be a field of prime characteristic $p$, and $G$ a finite group with $p$ dividing $|G|$. Furthermore, let $M$ and $N$ be two finitely generated left $kG$-modules. Then either $\Ext_{kG}^{n}(M,N) = 0$ for all $n > 0$, or there exist integers $d>0$ and $m_0$ 
such that $\Ext_{kG}^{n}(M,N) \ne 0$ for all $n \ge m_0$ with $n \equiv j \hspace{1mm} ( \mod d)$, for at least one $j$ with $0\le j<d$.
\end{theorem}

\begin{proof}
We denote the stable category of finitely generated left $kG$-modules by $\underline{\mod} kG$, and the group of morphisms between two objects $M,N$ by $\stHom_{kG}(M,N)$. Furthermore, we identify $\oplus_{n \ge 1} \stHom_{kG}(M, \Omega_{kG}^{-n}(N))$ with $\Ext_{kG}^{\ge 1}(M,N)$. The cohomology ring $\Coh^*(G,k)$ acts centrally on $\underline{\mod} kG$, as explained in the paragraph preceding this theorem. Moreover, the result of Evens and Venkov implies that $\oplus_{n \ge 0} \stHom_{kG}(M, \Omega_{kG}^{-n}(N))$ is a finitely generated $\Coh^*(G,k)$-module; in degree zero we have $\Coh^0(G,k) = k$, and $\stHom_{kG}(M,N)$ is finite-dimensional. 

Since condition (1) in the statement of Theorem \ref{thm:main} holds, we can now apply the theorem with $n_0 =0$: either $\Ext_{kG}^{n}(M,N) = 0$ for all higher $n$, or there exist integers $d>0$ and $m_0$ 
such that $\Ext_{kG}^{n}(M,N) \ne 0$ for all $n \ge m_0$ with $n \equiv j \hspace{1mm} ( \mod d)$, for at least one $j$ with $0\le j<d$. Finally, since the group algebra $kG$ is selfinjective, the arguments from Remark \ref{rem:findim}(1) can be applied so that the vanishing conclusion becomes $\Ext_{kG}^{n}(M,N) = 0$ for all $n > 0$.
\end{proof}

\begin{remark}\label{rem:grp}
Choose homogeneous elements $x_1, \dots, x_t \in \Coh^*(G,k)$, of positive degrees, generating the commutative even subalgebra $\Coh^{\ev}(G,k)$ of $\Coh^*(G,k)$ over $\Coh^0(G,k)$. It follows from Remark \ref{rem:main}(2) that we may take the $d$ in Theorem \ref{thm:grp} to be the least common multiple of the degrees of these generators. If the characteristic of the field $k$ is two, so that the whole cohomology ring $\Coh^*(G,k)$ is commutative, we may take generators for $\Coh^*(G,k)$ itself.
\end{remark}

Let us look at an example with some of the symmetric groups.

\begin{example}
Let $S_n$ denote the $n$th symmetric group and $k$ the field with two elements. Furthermore, let $M$ and $N$ be two finitely generated left $kS_n$-modules. Note that since we are in characteristic two, the group cohomology ring $\Coh^{*}(S_n,k)$ is commutative.

By \cite[Theorem 4.1]{Nakaoka1962}, the cohomology ring $\Coh^*(S_4,k)$ is generated over $\Coh^0(S_4,k)$ by elements $x,y,z$ in degrees $1, 2$ and $3$, respectively. The least common multiple of these degrees is $6$, and so from Theorem \ref{thm:grp} and Remark \ref{rem:grp} we obtain the following: if $\Ext_{kS_4}^{m}(M,N) \neq 0$ for some $m > 0$, then there exists a number $t \in \{ 0,1,2,3,4,5 \}$ such that $\Ext_{kS_4}^{m}(M,N) \neq 0$ for all $m \gg 0$ with $m \equiv t \hspace{1mm} ( \mod 6 )$.

In \cite{AdemMaginnisMilgram1990}, the cohomology ring $\Coh^*(S_n,k)$ was completely described for $n = 6,8,10,12$. It starts rather gently: as with $S_4$, the cohomology ring $\Coh^*(S_6,k)$ is generated over $\Coh^0(S_6,k)$ by elements of degrees $1,2$ and $3$. Here we obtain the same vanishing result: if $\Ext_{kS_6}^{m}(M,N) \neq 0$ for some $m > 0$, then there exists a number $t \in \{ 0,1,2,3,4,5 \}$ such that $\Ext_{kS_6}^{m}(M,N) \neq 0$ for all $m \gg 0$ with $m \equiv t \hspace{1mm} ( \mod 6 )$. However, already for $S_8$ it gets much more complicated. The cohomology ring $\Coh^*(S_8,k)$ is generated over $\Coh^0(S_8,k)$ by elements of degrees $1,2,3,4,5,6$ and $7$, and the least common multiple of these degrees is $420$. Therefore, we obtain the following: if $\Ext_{kS_8}^{m}(M,N) \neq 0$ for some $m > 0$, then there exists a number $t \in \{ 0,1, \dots, 419 \}$ such that $\Ext_{kS_8}^{m}(M,N) \neq 0$ for all $m \gg 0$ with $m \equiv t \hspace{1mm} ( \mod 420 )$.
\end{example}

\end{document}